\documentclass[11pt]{amsart}
\usepackage{fullpage}
\usepackage{graphicx}
\usepackage{amssymb,verbatim}
\usepackage{epstopdf}
\usepackage{amsmath,amsfonts,amsthm,amssymb,amscd,hyperref}
\usepackage{multirow}
\usepackage{pstricks}
\usepackage{tikz}
\usepackage{pgfplots}
\pgfplotsset{compat=1.14}
\usetikzlibrary{matrix}
\usepackage{enumerate}
\usepackage{algorithm}
\usepackage[noend]{algpseudocode}
\usepackage{cleveref}
\usepackage{subfigure}
\crefname{figure}{Figure}{Figures}
\tikzstyle{vertex}=[circle, draw, inner sep=0pt, minimum size=3pt]

\theoremstyle{plain}

\newtheorem{thm}{Theorem}[section]

\newtheorem{proposition}[thm]{Proposition}
\newtheorem{lem}[thm]{Lemma}
\newtheorem{problem}[thm]{Problem}
\newtheorem{cor}[thm]{Corollary}

\crefname{lem}{Lemma}{Lemmas}
\crefname{thm}{Theorem}{Theorems}

\theoremstyle{definition}
 
\newtheorem{definition}[thm]{Definition} 
\crefname{defn}{Definition}{Definitions} 
\newtheorem{ex}[thm]{Example}
 
\crefname{ex}{Example}{Examples}
\newtheorem{lemma}[thm]{Lemma}

\newtheorem{conjecture}{Conjecture} 

\theoremstyle{remark}
\newtheorem{remark}[thm]{Remark}

\newcommand{\Pin}{\textnormal{Pin}}
\newcommand{\Pk}{\textnormal{Pk}}
\newcommand{\Valley}{\textnormal{Vy}}
\newcommand{\V}{\textnormal{Vale}}
\newcommand{\vp}{\varphi}
\newcommand{\Orb}{\textnormal{Orb}}

\newcommand{\Z}{\mathbb{Z}}
\newcommand{\N}{\mathbb{N}}

\newcommand{\Sn}[1][n]{\mathfrak{S}_{#1}}

\newcommand{\VV}{\mathcal V}

\newcommand{\sS}{\mathfrak S}

\newcommand{\nPV}[1][PV]{N_{#1}}

\title{
A formula for enumerating permutations\\with a fixed pinnacle set
}

\author{Alexander Diaz-Lopez}
\address[A.~Diaz-Lopez]{Department of Mathematics \& Statistics\\
  Villanova University\\
  Villanova, PA 19085}
\email{\textcolor{blue}{\href{mailto:alexander.diaz-lopez@villanova.edu}{alexander.diaz-lopez@villanova.edu}}}

\author{Pamela E. Harris}
\address[P. E. Harris]{Department of Mathematics and Statistics, Williams College, United States}
\email{\textcolor{blue}{\href{mailto:peh2@williams.edu}{peh2@williams.edu}}}

\author{Isabella Huang}
\address[I. Huang]{Department of Mathematics and Statistics, Williams College, United States}
\email{\textcolor{blue}{\href{mailto:ih5@williams.edu}{ih5@williams.edu}}}

\author{Erik Insko}
\address[E. Insko]{Department of Mathematics\\ 
Florida Gulf Coast University\\ 
Fort Myers, Florida 33965}
\email{\textcolor{blue}{\href{mailto:einsko@fgcu.edu}{einsko@fgcu.edu}}}

\author{Lars Nilsen}
\address[L. Nilsen]{Department of Mathematics\\ 
Florida Gulf Coast University\\ 
Fort Myers, Florida 33965}
\email{\textcolor{blue}{\href{mailto:lnilsen0182@eagle.fgcu.edu}{lnilsen0182@eagle.fgcu.edu}}}

\begin{document}


\begin{abstract}
In 2017 Davis, Nelson, Petersen, and Tenner pioneered the study of pinnacle sets of permutations and asked whether there exists a class of operations, which applied to a permutation in $\sS_n$, can produce any other permutation with the same pinnacle set and no others. In this paper, we adapt a group action defined by Foata and Strehl to provide a way to generate all permutations with a given pinnacle set. From this we give a closed non-recursive formula enumerating permutations with a given pinnacle set. Thus answering a question  posed by Davis, Nelson, Petersen, and Tenner.
\end{abstract}

\maketitle

\section{Introduction}
Let $\N$ denote the set of the nonnegative integers. For $n\in\N$, let $[n] = \{1,2,\dots,n\}$ and let  $\sS_n$ denote the set of permutations $\pi=\pi_1\pi_2\cdots\pi_n$ of $[n]$. 
Given a permutation $\pi=\pi_1\pi_2\cdots \pi_n \in \sS_n$, throughout we let $\pi_0=\pi_{n+1}=\infty$.
A permutation $\pi=\pi_1\pi_2 \cdots \pi_n$ has 
a \textbf{descent} at index $i$ if $\pi_i>\pi_{i+1}$,
an \textbf{ascent} at index $i$ if $\pi_{i}<\pi_{i+1}$,
 a \textbf{peak} at index $i$ if $\pi_{i-1}<\pi_i>\pi_{i+1}$, and
 a \textbf{valley} at index $i$ if $\pi_{i-1}>\pi_i<\pi_{i+1}$.
Whereas, the value $\pi_i$ of a permutation is a \textbf{pinnacle} if $\pi_{i-1}<\pi_i>\pi_{i+1}$, and a \textbf{vale} if $\pi_{i-1}>\pi_i<\pi_{i+1}$. 
 Then the 
\textbf{peak set} of a permutation $\pi\in\sS_n$ is $\Pk(\pi)=\{i\in[n]:i\text{ is a peak of  }\pi\}$, the 
\textbf{pinnacle set} of $\pi$ is the set $$\Pin(\pi) = \{ i\in[n] : i \text{ is a pinnacle of }\pi\},$$ the \textbf{valley set} of $\pi$ is 
$\Valley(\pi)=\{i\in[n]:i\text{ is a valley of  }\pi\}$, and the \textbf{vale set} of $\pi$ is the set $$\V(\pi)=\{i\in[n]:i\text{ is a vale of  }\pi\}.$$
Note that we can also think of the pinnacle set as the image of the peak set under the function $\pi$, and we can think of the vale set as the image of the valley set under the function $\pi$. 
For example, the permutation $\pi=15264387$ has pinnacle set $\Pin(\pi)=\{5,6,8\}$, peak set $\Pk(\pi)=\{2,4,7\}$, vale set $\V(\pi) = \{1,2,3,7\}$ and valley set  $\Valley(\pi)=\{1,3,6,8\}$.

Although, the notions of pinnacles and peaks (vales and valleys) capture a sense of a rise and fall (fall and rise) in a permutation, they behave rather differently. To capture this difference, we consider $T\subset[n]$ and let $\Pk(T;n)=\{\pi\in\sS_n:\Pk(\pi)= T\}$, $\Pin(T;n)=\{\pi\in\sS_n:\Pin(\pi)=T\}$, and present some previous results in the study of peaks and pinnacles of permutations.  In 2013, Billey, Burdzy, and Sagan presented a result regarding the enumeration of permutations in $\sS_n$ with a specified $n$-admissible peak set $T$. That is, $T\subset[n]$ such that $\Pk(T)\neq\emptyset$. Their main result is as follows.

\begin{thm}[Billey, Burdzy, and Sagan 2013  \cite{peakpaper}]\label{thm:BBS}
If $T = \{ i_1 < \cdots < i_s \}$ is an $n$-admissible peak set, then 
\begin{align}
|\Pk(T;n)|=p(n) 2^{n-|T| -1}\label{eq:BBS}    
\end{align} 
where $p(n)$ is a polynomial depending on $T$ such that $p(m)$ is an integer for all integral $m$ and $\deg p(n) = i_s -1$. 
\end{thm}

In 2017,  Davis, Nelson, Petersen, and Tenner determined bounds for the number of permutations with a specified $n$-admissible pinnacle set $P\subset[n]$. That is, $P\subset[n]$ such that $\Pin(P)\neq\emptyset$. Their main result is as follows.

\begin{thm}[Davis, Nelson, Petersen, and Tenner 2017 \cite{pinn}] \label{pinpaper}
If $P$ is an admissible pinnacle set, then 
\[2^{n-|P|-1}\;\leq \;|\Pin(P;n)|\;\leq\; |P|! \cdot 2^{n-2|P|-1} \cdot \mathcal S(n-|P|, |P|+1)\] 
where $\mathcal S(r, s)$ denotes the Stirling partition number, which counts the number of ways to partition a set of $r$ objects into $s$ non-empty subsets. Moreover, these bounds are sharp. \end{thm}

Davis et. al. posed the question of whether there exists a class of operations, which applied to a permutation in $\sS_n$, can produce any other permutation with the same pinnacle set and no others \cite[Question 4.2]{pinn}.
In this paper, we provide a way to generate all permutations with a given pinnacle set by using a group action on permutations called the dual Foata-Strel action, which we define in Section \ref{sec:FSfacts}. Specifically, this action partitions the set $\Pin(P;n)$ into disjoint orbits, and we generate one permutation in each orbit.  From this we then provide a closed non-recursive formula for the total number of permutations with a given pinnacle set. Thereby answering \cite[Question 4.4]{pinn}. To state our result, for
a given pinnacle set $P$, we define $\VV(P)$ to be the set of all vale sets, $V\subseteq([n]\setminus P)$, so that $P$ and $V$ are an $n$-admissible pinnacle and vale set combination, i.e. there are permutations in $\Sn$ with $P$ as their pinnacle set and $V$ as their vale set.

\begin{thm}\label{thm:main result}
If $P$ is an $n$-admissible pinnacle set, then 
\[\Pin(P;n)=2^{n-|P|-1}\sum_{V\in \VV(P)}\prod_{p \in P}\binom{\nPV(p)}2\prod_{x \in [n] \setminus (P \ \cup\ V)} \nPV(x),\]
where $V_k = \{v \in V : v < k\}$, $P_k = \{p \in P : p < k\}$, and $\nPV(k) = |V_k| - |P_k|$, counting the number of vales less
than $k$, minus the number of pinnacles less than $k$.
\end{thm}

This work is organized as follows. In Section \ref{sec:FSfacts} 
we define the dual Foata-Strehl group action on permutations, recall some known characteristics of this action, and establish that the dual Foata-Strehl group action on permutations preserves pinnacle sets (Theorem~\ref{pinpreserve}). 
In Section \ref{sec:Rep} we describe a unique representative from each orbit under the dual Foata-Strehl action (Theorem~\ref{thm:uniqueFSminimal}). In Section \ref{sec:count} we construct and count permutations with a fixed pinnacle set, culminating in a proof of Theorem \ref{thm:main result}. In Section \ref{sec:code} we present computational evidence that the algorithm based on our constructions in Section \ref{sec:count} is drastically faster than the naive algorithm for generating $\Pin(P;n)$. In Section \ref{sec:future} we present a few open problems for further study.

\section{The dual Foata-Strehl group action on $\sS_n$}\label{sec:FSfacts}

Let $\pi \in \sS_n$ and $x\in [n]$. We can write $\pi = w_1 w_2 x w_4 w_5$ where $w_2$ is the longest contiguous subword immediately to the left of $x$ such that all values are less than $x$ and $w_4$ is the longest contiguous subword immediately to the right of $x$ such that all letters of $w_4$ are less than $x$. Call this the $x${\it-factorization of }$\pi$, then let $\varphi_x(\pi) = w_1w_4xw_2w_5$, which defines an involution on $\sS_n$. Note that if $x$ is a vale, then $w_2=\emptyset=w_4$, where $\emptyset$ denotes the empty word, and $\varphi_x(\pi)=\pi$.

The map $\varphi_x$ is a modified version of the map Foata and Strehl defined in \cite{FS}. In their paper, the $x$-factorization of $w$ was defined by letting $w_2$ be the longest contiguous subword immediately to the left of $x$ such that all values are greater than $x$ and $w_4$ is the longest contiguous subword immediately to the right of $x$ such that all letters of $w_4$ are greater than $x$. Then they use their $x$-factorization to define the map $\phi_x(\pi)=\phi_x(w_1w_2 x w_4w_5)=w_1w_4xw_3w_5$. 

In that sense, $\varphi_x$ and $\phi_x$ only differ in that one switches the values near $x$ that are less than $x$ and the other switches the values near $x$ that are greater than $x$. If we let $w_0$ be the longest word of $\mathfrak{S}_n$, namely $w_0=n\,(n-1)\, \cdots 1$, and if $\pi=\pi_1\pi_2\cdots \pi_n$ then $w_0(\pi_i)=n-\pi_i+1$ for all $1\leq i\leq n$. Hence, for any $x\in[n]$, we have that
\begin{align}\label{eq:FT to us}
    \varphi_x(\pi)&=w_0(\phi_{w_0(x)}(w_0\pi)).
\end{align}
Geometrically, this equation states that to obtain $\varphi_x(\pi)$ we can first flip $\pi$ vertically along the $y=n/2$ line, which is achieved by multiplying $\pi$ by $w_0$ on the left. Then, we  apply the map $\phi_{w_0(x)}$, and finally flip the permutation vertically again along the same line.

\begin{ex} If $\pi = 6534127$, then 

\begin{align*}
\varphi_4(\pi)&=\varphi_4(\underbrace{65}_{w_1}\underbrace{3}_{w_2}4\underbrace{12}_{w_4}\underbrace{7}_{w_5})=6512437\\
\varphi_5(\pi)&=\varphi_4(\underbrace{6}_{w_1}\underbrace{\emptyset}_{w_2}5\underbrace{3412}_{w_4}\underbrace{7}_{w_5})=6341257,
\end{align*}
and 
\begin{align*}
\phi_4(\pi)&=\varphi_4(\underbrace{653}_{w_1}\underbrace{\emptyset}_{w_2}4\underbrace{\emptyset}_{w_4}\underbrace{127}_{w_5})=6534127\\
\phi_5(\pi)&=\varphi_4(\underbrace{\emptyset}_{w_1}\underbrace{6}_{w_2}5\underbrace{\emptyset}_{w_4}\underbrace{34127}_{w_5})=5634127.
\end{align*}
Repeating this process shows that $\varphi_5(\varphi_4(\pi)) = 6124357=\varphi_4(\varphi_5(\pi))$ and $\phi_5(\phi_4(\pi))=5634127=\phi_4(\phi_5(\pi))$. 
\end{ex}

\begin{lem}\label{commute} If $x,y\in [n]$, then $\varphi_x(\varphi_y(\pi)) = \varphi_y(\varphi_x(\pi))$ and $\phi_x(\phi_y(\pi)) = \phi_y(\phi_x(\pi))$ for any $\pi \in \sS_n$. \end{lem}

\begin{proof}
In Section 2 of \cite{FS} Foata and Strehl prove that for any $x,y\in[n]$ and any permutation $\pi$ we have that $\phi_x(\phi_y(\pi))=\phi_y(\phi_x(\pi))$. Thus,
\begin{center}
\begin{tabular}{rll}
    $\varphi_x(\varphi_y(\pi))$&$=\varphi_x\big(w_0(\phi_{w_0(y)}(w_0\pi))\big)$&by \eqref{eq:FT to us} applied to $\varphi_{y}$\\
    &$=w_0\Big(\phi_{w_0(x)}\Big(w_0\big(w_0\phi_{w_0(y)}(w_0\pi)\big)\Big)  \Big)$&by \eqref{eq:FT to us} applied to $\varphi_{x}$\\
    &$=w_0\Big(\phi_{w_0(x)}\Big(\phi_{w_0(y)}(w_0\pi)\Big)  \Big)$&as $w_0$ is an idempotent\\
    &$=w_0\Big(\phi_{w_0(y)}\Big(\phi_{w_0(x)}(w_0\pi)\Big)  \Big)$&since $\phi_{w_0(x)}$ and $\phi_{w_0(y)}$ commute\\
    &$=w_0\Big(\phi_{w_0(y)}\Big(w_0\big(w_0\phi_{w_0(x)}(w_0\pi)\big)\Big)$&as $w_0$ is an idempotent\\
    &$=\varphi_y\big(w_0(\phi_{w_0(x)}(w_0\pi))\big)$&by \eqref{eq:FT to us} applied to $\varphi_{y}$\\
    &$=\varphi_y(\varphi_x(\pi))$& by \eqref{eq:FT to us} applied to $\varphi_{x}$.
\end{tabular}
\end{center}
\end{proof}

Given $S\subseteq[n]$, Foata and Strehl \cite{FS} define 
$$\phi_S(\pi) = \prod_{x\in S} \phi_x(\pi)$$ where the product notation denotes the composition of the functions $\phi$ for all $x\in S$, and if $S=\emptyset$, then $\phi_S$ is the identity map on $\mathfrak{S}_n$. Since $\phi_x$ and $\phi_y$ commute for all $x,y\in [n]$, then $\phi_S(\pi)$ is well defined. This can be interpreted as a group action $\phi:\Z_2^n\times \sS_n\to\sS_n$ defined by
$\phi(\textbf{a},\pi)=\phi_{X_{\textbf{a}}}(\pi)$ where $X_\textbf{a}:=\{i:a_i=1\}$. We call $\phi_S$ the Foata-Strehl action.

Given $S\subseteq[n]$, we can similarly define 
$$\varphi_S(\pi) = \prod_{x\in S} \varphi_x(\pi)$$
where the product notation denotes the composition of the functions $\varphi_x$ for all $x\in S$. When $S=\emptyset$, define $\varphi_S$ to be the identity map on $\sS_n$. Since $\varphi_x$ and $\varphi_y$ commute for all $x,y\in[n]$, then $\varphi_S$ is well defined. Similarly, the group $\Z_2^n$ acts on the symmetric group $\sS_n$ via the function $\varphi_S$.  
To be precise, $\varphi:\Z_2^n\times \sS_n\to\sS_n$ defined by
$\varphi(\textbf{a},\pi)=\varphi_{X_{\textbf{a}}}(\pi)$ where $X_\textbf{a}:=\{i:a_i=1\}$ is a group action. We henceforth refer to  $\varphi_S$ as the dual Foata-Strehl action.

Our first result establishes that the dual Foata-Strehl action preserves the pinnacle set of a permutation. 

\begin{thm}\label{pinpreserve} For any $S\subset [n]$ and any $\pi \in \sS_n$, $\rm{Pin}(\pi) = \rm{Pin}(\varphi_S(\pi))$. \end{thm}

\begin{proof}
First, note that it is enough to show that for any $x\in[n]$, $\Pin(\pi) = \Pin(\varphi_x(\pi))$. We write 
\begin{eqnarray} 
\pi &=& w_1 w_2 x w_4 w_5 \label{eq1}\\ 
\varphi_x(\pi) &=& w_1 w_4 x w_2 w_5. \label{eq2}
\end{eqnarray}

Since  $\varphi_x$ is an involution, it is enough to prove that $\Pin(\pi) \subset \Pin(\varphi_x(\pi))$. Let $y \in \Pin(\pi)$. We show $y\in \Pin(\varphi_x(\pi))$. First, consider the case that $y=x$. It is clear that $y$ will still be a pinnacle of $\varphi_x$, as the subwords $w_2$ and $w_4$ are defined to be strictly smaller than $y$, and swapping the two words around $y$ will preserve the fact that $y$ is still a~pinnacle. 

Now, we consider the cases such that the pinnacle $y$ is contained in subwords $w_1$, $w_2$, $w_4$, or $w_5$. Note that it is enough to consider the cases when $y$ is on one of the ends of the words that comprise the factorization. Indeed, since the action preserves the structure of the subwords themselves, changes to the pinnacle set will only arise at the junctions between the subwords.

We first consider the case that $y \in w_1$. Since the left-most letter of $w_1$ cannot be a pinnacle by definition, we consider the case when $y$ is the right-most letter of $w_1$. If $y$ is a pinnacle of $\pi$, it must be greater than its neighbor to the left in $w_1$, which remains the same in $\varphi_x(\pi)$. Since $y\in w_1$, it must be the case that $y > x$, by definition of the $x$-factorization. Moreover, all letters of $w_4$ will be less than $x$, which is also less than $y$. Furthermore, if $w_4$ is empty, then the neighbor to the right of $y$ in $\varphi_x(\pi)$ is $x$ itself. So, $y$ is always greater than its neighbor to the right in $\varphi_x(\pi)$ and is thus a pinnacle. This argument similarly applies to the case that $y$ is the left-most letter of $w_5$. 

We claim that it is impossible to have a pinnacle on the ends of $w_2$ and $w_4$. We consider the case of $w_2$ and note that an analogous  argument applies to $w_4$. Suppose $y$ is a pinnacle sitting at the right-most end of $w_2$. Because it is a pinnacle, $y$ must be greater than its neighbor to the right, namely $x$. However, by definition of the $x$-factorization, the letter $y$ would not be in $w_2$, as $w_2$ is the longest contiguous word to the left of $x$ whose letters are all \emph{less} than $x$. Now suppose $y$ is a pinnacle sitting at the left-most end of $w_2$. By definition of pinnacle, $y$ must be greater than its left neighbor, which is in $w_1$. On the other hand, by definition of the $x$-factorization, all letters of $w_2$ -- and thus $y$ -- are less than $x$, and the neighbor to the left of $y$ in $w_1$ must be greater than $x$ and thus greater than $y$. We have arrived at a contradiction and conclude that the left-most end of $w_2$ cannot be a~pinnacle. 

Thus, we have shown that $\Pin(\pi) \subset \Pin(\varphi_x(\pi))$, which implies that $\Pin(\pi) = \Pin(\varphi_x(\pi))$ for arbitrary $x \in [n]$. Thus, we can conclude that the dual Foata-Strehl action preserves pinnacle sets. 
\end{proof}

Let $\thicksim$ be the equivalence relation on $\mathfrak{S}_n$ defined by the action of $\mathbb{Z}_2^n$. Namely, $\pi \thicksim \tau$ if and only if there exists $\textbf{a}\in \Z_2^n$ such that $\varphi_{X_{\textbf{a}}}(\pi) = \tau$. 
The equivalence classes under this relation are precisely the orbits of the dual Foata-Strehl action.
In light of Theorem \ref{pinpreserve}, we know that these orbits partition $\sS_n$ into subsets of permutations sharing a pinnacle~set. 

The following example illustrates the there may be multiple equivalence classes with the same pinnacle~set.

\begin{ex}
In Table \ref{tab:1}, each row represents an equivalence class of $\sS_4$ arising from the dual Foata-Strehl action, and we have labeled the pinnacle set of each class at the left of the row. Note that there are three equivalence classes with the same pinnacle set $P=\{4\}$.
\begin{table}[h]
\begin{tabular}{|c||cccccccc|}\hline
Pinnacle set&\multicolumn{8}{c|}{Equivalence class}\\\hline\hline
$P=\emptyset$&1234&2134&3124&4123&3214&4213&4312&4321\\\hline
$P=\{3\}$&1324&2314&4132&4231&&&&\\\hline
$P=\{4\}$&1243&2143&3412&3421&&&&\\\hline
$P=\{4\}$&1342&3142&2413&2431&&&&\\\hline
$P=\{4\}$&1423&1432&2341&3241&&&&\\\hline
\end{tabular}
\caption{Partitioning of $\sS_4$ by the dual Foata-Strehl action.}\label{tab:1}
\end{table}
\end{ex}

Next we measure the size of each equivalence class and do so by examining the relationship between pinnacles and vales of permutations. 

In what follows we let $v(\pi)$ denote the number of vales in $\pi$.

\begin{lem}\label{numval} If $P$ is an $n$-admissible pinnacle, then $v(\pi)=|P| + 1$ for all $\pi\in \Pin(P;n)$. \end{lem}

\begin{proof}
Since $\pi_0=\pi_{n+1}=\infty$, and since vales and pinnacles alternate we know there will be one more vale than pinnacles.
\end{proof}

For any $\pi \in \sS_n$, let $\Orb_\varphi(\pi):=\{\vp_S(\pi):S\subseteq[n]\}$ denote the orbit of $\pi$ under the dual Foata-Strehl action $\varphi$. Similarly, let $\Orb_\phi(\pi):=\{\phi_S(\pi):S\subseteq[n]\}$ denote the orbit of $\pi$ under the Foata-Strehl action $\phi$. In \cite[Section 3]{FS}, Foata-Strehl proved that
$$|\text{Orb}_\phi(\pi)|=2^{n-v(\pi)}.$$
We now prove the analogous result for $\Orb_\varphi(\pi)$. 

\begin{thm}\label{thm:orbitsize}
If $\pi \in \sS_n$, then $|\Orb_\varphi(\pi)|=2^{n-v(\pi)}$. 
\end{thm}
\begin{proof}
For a set $S\subseteq [n]$, let $w_0(S)=\{w_0(s)\, : \, s\in S\}$, where $w_0(s)=n-s+1$. We now create a bijection between $\text{Orb}_\varphi(\pi)$ and $\text{Orb}_\phi(w_0\pi)$. Let $$F:\text{Orb}_\varphi(\pi) \to \text{Orb}_\phi(w_0\pi) \text{ such that } F(\varphi_S(\pi))=\phi_{w_0(S)}(w_0\pi)$$
and 
$$G:\text{Orb}_\phi(w_0\pi) \to \text{Orb}_\varphi(\pi) \text{ such that } G(\phi_S(w_0\pi))=\varphi_{w_0(S)}(\pi).$$
Then $F\circ G$ and $G\circ F$ are the identity maps on $\text{Orb}_\phi(w_0\pi)$ and $\text{Orb}_\varphi(\pi)$, respectively. Thus, 
\[|\text{Orb}_\varphi(\pi)|=|\text{Orb}_\phi(w_0\pi)|=2^{n-v(\pi)}.\qedhere\]
\end{proof}

We remark that Foata and Strehl determined that the number of orbits under $\phi$  is given by the $n$-th tangent or secant number, depending on whether $n$ is odd or even. By Theorem \ref{thm:orbitsize}, the same is true for the number of orbits under $\varphi$. We would now like to count the number of orbits of $\varphi$ that have a prescribed pinnacle set $P$. This is the content of the subsequent sections. 

\begin{remark}
 In \cite{valleyhop}, Petter Br\"and\'en defined a modified function,  that we call $\varphi'_x$, such that $\varphi_x'(\pi)=\varphi_x(\pi)$ if $x$ is neither a pinnacle nor a vale, and $\varphi'_x(\pi)=\pi$ if $x$ is a pinnacle or a vale. Similar to $\phi$ and $\varphi$, the author defines $\varphi'$ as an action of $\Z_2^n$ on $\sS_n$ and uses it to prove that for any $T\subseteq \sS_n$, the polynomial defined by
$$A(T;x)=\sum_{\pi \in T}x^{\text{des}(\pi)}$$
is $\gamma$-nonnegative, where $\text{des}(\pi)=|\{i \in [n]\,|\,\pi_i>\pi_{i+1}\}|.$
In \cite{gal},  Postnikov, Reiner and Williams defined a modified function,  that we call $\varphi''$, such that $\varphi_x''(\pi)=\phi_x(\pi)$ if $x$ is neither a pinnacle nor a vale and $\varphi''(\pi)=\pi$ if $x$ is a pinnacle or a vale. Similar to $\phi,\varphi,$ and $\varphi'$, they define an action $\varphi''$ of $\Z_2^n$ on $\sS_n$ and use it to prove Gal's conjecture for the chordal nestohedra, \cite[Theorem 11.6]{gal}.
\end{remark}

\section{Representatives of dual Foata-Strehl orbits} \label{sec:Rep}
In this section, we describe a collection of permutations, called \textit{$FS$-minimal permutations}, that characterize the orbits of the dual Foata-Strehl action $\varphi$. Then in Section \ref{sec:count}, we provide a construction of all $FS$-minimal permutations with a given pinnacle set. These results will allow us to count all permutations with a given pinnacle set.

\begin{definition}[Admissibility] \label{def:adms}
A pair of sets $(P,V)$ is considered \textbf{admissible} if there is a permutation with pinnacle set $P$ and vale set $V$. Given a pinnacle set $P$, define $\VV(P)$ to be the set of all vale sets $V$ for which the pair $(P,V)$ is admissible.
\end{definition}

Throughout the section, let $\pi$ be a permutation with pinnacle set $P=\{ p_1,\ldots, p_\ell \}$ and vale set $V=\{ v_1,\ldots, v_{\ell+1} \}$, respectively.
We will often list the pinnacles and vales in the order in which they appear in $\pi$, from left to right. 
We will also restrict $\pi$ to permutations of the sets $P$, $V$, and  $P \cup V \subseteq [n]$.  For instance, we write 
$\pi|_{P} = p_1p_2\cdots p_\ell$ to denote the restriction of the permutation $\pi$ to just the values at which $\pi$ has pinnacles, which we list in the order they appear in $\pi$. Similarly, $\pi|_{V} =v_1v_2\cdots v_{\ell+1}$  denotes the restriction of the permutation $\pi$ to just the values at which $\pi$ has vales, which we list in the order they appear in $\pi$. Similarly, we let
$$\pi|_{P\cup V}=v_1p_1v_2p_2\cdots p_\ell v_{\ell+1}$$ denote the restriction of $\pi$ to just the values at which $\pi$ has vales and pinnacles, listed in the order they appear in $\pi$.
For example, if $\pi=32814756$, then $P=\{7,8\}$, $V=\{1,2,5\}$, $\pi|_{P}=87$, $\pi|_{V}=215$, and $\pi|_{P\cup V}=28175$.

In what follows, we present three technical lemmas used to prove the main theorem of the section, Theorem \ref{thm:uniqueFSminimal}.

 \begin{lemma}\label{lem:FSdescents}
    If $\pi$ is a permutation with pinnacle set $P=\{p_1,\ldots, p_{\ell}\}$, then for all $i \in [\ell]$, $\pi$ and $\varphi_{p_i}(\pi)$ have the same number of descents.
    \end{lemma}
    \begin{proof}
    For any $i \in [\ell]$, consider the $p_i$-factorization of $\pi$, $$\pi=w_1w_2p_iw_4w_5=\underbrace{\pi_1\cdots \pi_{k_1}}_{w_1}\underbrace{\pi_{k_1+1}\cdots \pi_{k_2}}_{w_2}p_i \underbrace{\pi_{k_4}\cdots \pi_{k_5-1}}_{w_4}\underbrace{\pi_{k_5}\cdots \pi_n}_{w_5}.$$
    By the definition of this factorization $p_i>\max(w_2)$, $p_i>\max(w_4)$ and $\pi_{k_1}>p_i<\pi_{k_5}$. Applying $\varphi_{p_i}$  we get  
    $$\varphi_{p_i}(\pi)=w_1w_4p_iw_2w_5=\underbrace{\pi_1\cdots \pi_{k_1}}_{w_1}\underbrace{\pi_{k_4}\cdots \pi_{k_5-1}}_{w_4}p_i\underbrace{\pi_{k_1+1}\cdots \pi_{k_2}}_{w_2}\underbrace{\pi_{k_5}\cdots \pi_n}_{w_5}.$$
     Since the content in $w_1,w_2,w_4,w_5$ did not change, it is enough to study the places where these subwords meet in $\varphi_{p_i}(\pi)$, namely $\pi_{k_1}\pi_{k_4}, \pi_{k_5-1}p_i, p_i\pi_{k_1+1}$, and $\pi_{k_2}\pi_{k_5}$.  Since $\pi_{k_4}<p_i<\pi_{k_1}$ and $p_i>\pi_{k_1+1}$, the descents $\pi_{k_1}\pi_{k_1+1}$ and $p_i\pi_{k_4}$ in $\pi$ got replaced by the descents $\pi_{k_1}\pi_{k_4}$ and $p_i\pi_{k_1+1}$ in $\varphi_{p_i}(\pi)$, respectively. Similarly, the ascents $\pi_{k_2}p_i$ and $\pi_{k_5-1}\pi_{k_5}$ in $\pi$ got replaced by the ascents $\pi_{k_5-1}p_i$ and $\pi_{k_2}\pi_{k_5}$ in $\varphi_{p_i}(\pi)$, respectively. Thus, the number of descents remained constant.
    \end{proof}

\begin{lemma} \label{lemma:xy}
Let $\pi\in\sS_n$ and let $x,y$ be two distinct elements in $[n]$.  If \begin{align*}
    &v_1\, v_2\, x \,v_4\,v_5\quad\mbox{ is the $x$-factorization of $\pi$ and}\\
    &\alpha_1 \alpha_2  x \alpha_4 \alpha_5\quad\mbox{is the $x$-factorization of $\varphi_{y}(\pi)$,}
\end{align*} 
then  
$\max(v_2)=\max(\alpha_2) \text{ and } \max(v_4)=\max(\alpha_4).$
\end{lemma}
\begin{proof}
Let $x,y$ be two distinct elements in $[n]$. Let \begin{align*}
    v_1\, v_2\, x\, v_4\,v_5&\quad\mbox{denote the $x$-factorization of $\pi$,}\\
    w_1\, w_2\, y\, w_4\,w_5&\quad\mbox{denote the $y$-factorization of $\pi$,}\\
    \alpha_1 \,\alpha_2 \,x\, \alpha_4\,\alpha_5&\quad\mbox{denote the $x$-factorization of $\varphi_y(\pi)$, and }\\
    \beta_1\, \beta_2\, y\, \beta_4\,\beta_5&\quad\mbox{denote the $y$-factorization of $\varphi_{y}(\pi)$.}
\end{align*}
There are six possible cases to consider.  In the first four cases, detailed below, the subword $v_2 x v_4$ remains unchanged in $\varphi_{y}(\pi)$, hence $v_2 x v_4=\alpha_2 x \alpha_4$.
\begin{enumerate}
\item If $v_2 x v_4$ lies in $w_2$ then the subword $v_2 x v_4$ remains together, but is moved to $\beta_4$ in $\varphi_{y}(\pi)$. In this case
$v_2 x v_4 = \alpha_2 x \alpha_4 $.
\item If $v_2 x v_4$ lies in $w_4$ then $v_2 xv_4$ remains together, but is moved to $\beta_2$ in $\varphi_{y}(\pi)$. In this case $v_2 x v_4 = \alpha_2 x \alpha_4 $.
\item If $v_2 x v_4$ lies in $w_1$ then $v_2 x v_4$ remains in $\beta_1$ in $\varphi_{y}(\pi )$. In this case $v_2 x v_4 = \alpha_2 x \alpha_4 $.
\item If $v_2 x v_4$ lies in $w_5$ then $v_2 x v_4$ remains together in $\beta_5$ in $\varphi_{y}(\pi)$.
\end{enumerate} 
In the last two cases, described below, either $v_2$ or $v_4$ is rearranged slightly in $\alpha_2$ or $\alpha_4$, but this does not affect the maximum element of $\alpha_2$ or $\alpha_4$ in $\varphi_{y}(\pi)$.
\begin{enumerate}
\item[(5)] {If $v_2 x$  lies in $w_1$ but $v_4$ does not lie entirely in $w_1$  then $y$ is contained in $v_4$.  
In this case the subword  $v_2 x  $ remains unchanged in $\varphi_{y}(\pi)$ in the sense that $v_2 x = \alpha_2 x $, and 
$v_4$ has some of its elements rearranged by $\varphi_{y}$ but the set of elements appearing in $\alpha_4$ remains the same (i.e. $v_2=\alpha_2$ and the underlying set of $v_4$ is equal to the underlying set of $\alpha_4$). Hence 
$\max(v_2)=\max(\alpha_2) \text{ and } \max(v_4)=\max(\alpha_4)$ in this case. 
}
\item[(6)] {If $xv_4$  lies in $w_5$ but $v_2$ does not lie entirely in $w_5$, then $y$ is contained in $v_2$.  
In this case the word  $xv_4$ remains unchanged in $\varphi_{y}(\pi)$ in the sense that $xv_4=x\alpha_4$, and 
$v_2$ has some of its elements rearranged by $\varphi_{y}$ but the set of elements appearing in $v_2$ remains the same. Hence 
$\max(v_2)=\max(\alpha_2) \text{ and } \max(v_4)=\max(\alpha_4 )$ in this case. }\qedhere
\end{enumerate}
\end{proof}

We can also define an $x$-factorization of any subword of a permutation.  That is, given a subword $\sigma=s_1s_2\cdots s_\ell$ of a permutation $\pi\in\sS_n$, and $x=s_i$ for some $1\leq i\leq \ell$, then the $x$-factorization of $\sigma$ is $w_1w_2xw_4w_5$ where  $w_2$ is the longest contiguous subword immediately to the left of $x$ such that all values are less than $x$ and $w_4$ is the longest contiguous subword immediately to the right of $x$ such that all letters of $w_4$ are less than $x$. We then define $\varphi_x(\sigma)$ to be 
$$\varphi_x(\sigma)=w_1w_4 xw_2w_5.$$

\begin{lemma}\label{lem:FS restrict=restrict FS}
Let $\pi$ be a permutation with pinnacle set $P$ and vale set $V$. If $\pi|_P=p_1p_2\cdots p_\ell$ and  $\pi|_V=v_1v_2 \cdots v_{\ell+1}$, then $\vp_{p_i}(\pi)|_{P\cup V} = \vp_{p_i}(\pi|_{P\cup V})$ for any $i \in [\ell]$.
\end{lemma}
\begin{proof}
Fix a pinnacle $p_i$ in $\pi$ and consider the $p_i$-factorization
\(\pi=w_1w_2p_iw_4w_5.\)
Now define $\alpha_i=w_i|_{P\cup V}$ for each $i=1,2,3,4$. Then,
\begin{align*}
    \varphi_{p_i}(\pi)|_{P\cup V}&=(w_1\,w_4\, p_i\, w_2\, w_5)|_{P\cup V}\\
    &=w_1|_{P\cup V}\; w_4|_{P\cup V}\; p_i \; w_2|_{P\cup V}\; w_5|_{P\cup V}\\
    &=\alpha_1\; \alpha_4\; p_i \;\alpha_2\; \alpha_5.
\end{align*}

In the case where neither $w_1$ nor $w_5$ are empty, suppose $p'$ is the right most pinnacle in $w_1$, and $p''$ is the left most pinnacle in $w_5$. Hence, $p'>p_i$ and $p''>p_i$. 
Now consider 
\[\pi|_{P\cup V}=v_1\, p_1\, v_2\, p_2\,\cdots\, v_i\, p_i\, v_{i+1}\,\cdots\, p_\ell\, v_{\ell+1}.\]
Let $\pi|_{P\cup V}=w_1'w_2'p_iw_4'w_5'$ be the $p_i$-factorization of $\pi|_{P\cup V}$. Since $p'>p_i$ and $p''>p_i$, then $p'\in w_1'$ and $p''\in w_5'$. It now follows that $w_1'=w_1|_{P\cup V}=\alpha_1$ and $w_5'=w_5|_{P\cup V}=\alpha_5$. Thus, the $p_i$-factorization of $\pi|_{P\cup V}$ is \[\pi|_{P\cup V}=\alpha_1\alpha_2p_i\alpha_4\alpha_5.\]
Therefore 
\[    \varphi_{p_i}(\pi|_{P\cup V})=\alpha_1 \alpha_4 p_i  \alpha_2 \alpha_5=  \varphi_{p_i}(\pi)|_{P\cup V}.\]

Note that $w_2,w_4$ cannot be empty as $p_i$ is a pinnacle, and so the  proof is complete by noting that whenever $w_1$ or $w_5$ are empty, it implies $w_i=w_i'=\alpha_i=\emptyset$ for $i=1,5$, respectively. 
\end{proof}

We now define the notion of $FS$-minimal permutations and proceed to show that there is a unique $FS$-minimal permutation in each dual Foata-Strehl orbit of $\sS_n$.   
\begin{definition}
 A permutation $\pi$ is \textbf{FS-minimal} if 
 $\pi$ contains no double descents and 
 for each $p \in \Pin(\pi)$ the $p$-factorization $w_1\, w_2\, p\, w_4\, w_5$ of $\pi|_{P\cup V}$ satisfies $\max(w_2)<\max(w_4)$.
\end{definition}
 
\begin{thm}\label{thm:uniqueFSminimal}
If $\pi$ is a permutation with pinnacle set $P$ and vale set $V$, then there is a unique $FS$-minimal permutation in the dual Foata-Strehl orbit of $\pi$.
\end{thm}

\begin{proof}
We first show there is an $FS$-minimal permutation in each orbit and then show this permutation is unique. Let $\pi$ be a permutation with $$\pi|_P=p_1p_2\cdots p_\ell, \quad \pi|_V=v_1v_2\ldots v_{\ell+1},\quad \text{ and } \quad \pi|_{P\cup V} = v_1p_1v_2p_2\cdots v_\ell p_\ell v_{\ell+1}.$$ Let $$R=\{r \in [n]\setminus (P \cup V): r\ \mbox{appears  left of $v_1$ or between  $p_k$ and $v_{k+1}$ for some $1\leq k\leq \ell$}\},$$
that is $r$ is either in the beginning descending segment of $\pi$ or in a descending segment strictly between a pinnacle and a vale. Note that this implies that $\pi$ has $|P|+|R|$ descents. 
The $r$-factorization of $\pi$ is then
   $w_1\emptyset r w_4 w_5$ 
and 
\[\varphi_r(\pi)=w_1 w_4 r\emptyset w_5.\]
In $\varphi_r(\pi)$ we solely moved $r$ from  a descending segment to an ascending segment and left the rest of $\pi$ unchanged. Hence,  $ \varphi_r(\pi)$ has one fewer descent than $\pi$, since the relative order of the entries in $w_4$ remains unchanged. Then let $\rho(\pi):=\prod_{r\in R}\varphi_r(\pi)$.
By this construction, $\rho(\pi)$ has only $|P|$ descents occurring only at the indices of pinnacles (at the peak set of $\pi$), and none of these descents occur consecutively, i.e. there are no double descents.  
    
Let $$T=\{p\in P: \mbox{ the $p$-factorization $w_1 w_2 p w_4 w_5$ of } \pi|_{P\cup V} \mbox{ satisfies } \max(w_2)>\max(w_4) \},$$ 
and define $\tau(\rho(\pi)):=\prod_{t\in T}\varphi_t(\rho(\pi))$.     {We now claim that  $\tau(\rho(\pi))$ is $FS$-minimal.      Since $\rho(\pi)$ has no double descents, then by Lemma \ref{lem:FSdescents}, $\tau(\rho(\pi))$ has no double descents. }
    
Let $t \in T.$ By definition of the dual Foata-Strehl action, $\varphi_t(\rho(\pi))$ satisfies that $\max(w_2)<\max(w_4)$ in the $t$-factorization of $\varphi_t(\rho(\pi))|_{P\cup V}$. By Lemma \ref{lemma:xy}, for all other pinnacles $p \in P$, applying $\varphi_t$ to $\rho(\pi)$ does not change $\max(w_2)$ nor $\max(w_4)$ in the $p$-factorization of $\rho(\pi)$. Repeating this argument for all other elements of $T$ and using the fact that by Lemma \ref{lem:FS restrict=restrict FS} we can apply the dual Foata-Strehl action and then restrict to $P\cup V$ or restrict to $P\cup V$ and then apply the dual Foata-Strehl action and the result is the same, shows that $\tau(\rho(\pi))$ is $FS$-minimal.
    
To show this permutation is unique,  
suppose $\pi$ and $\sigma$ are both FS-minimal and lie in the same dual Foata-Strehl orbit.
Then $\pi = \vp_S(\sigma)$, for some $S \subseteq [n]$.
We will show that $S \subseteq V$, and since $\varphi_v(\sigma)=\sigma$ for all $v \in V$, then  $\pi=\vp_S(\sigma)= \sigma$.

Suppose $p \in P$. We will first show $P\cap S=\emptyset$. 
If $w_1'w_2'pw_4'w_5'$ is the $p$-factorization of $\vp_p(\sigma)|_{P\cup V}$, then $\max(w_2') > \max(w_4')$, since $\sigma$ is FS-minimal.
Lemma \ref{lemma:xy} shows that for any $k\in S$, applying $\vp_k$ to $\vp_p(\sigma)|_{P\cup V}$ would not change this inequality, thus $p \not \in S$ as otherwise $\max(w_2')>\max(w_4')$ in the $p$-factorization of $\pi$, contradicting that it is $FS$-minimal. Hence, $P\cap S=\emptyset$.

Since applying the dual Foata-Strehl action at a vale leaves a permutation unchanged, it suffices to show  $S\cap \left([n] \setminus (P \cup V)\right)=\emptyset$ to conclude $S\subset V$ and $\pi=\sigma_S(\sigma)=\sigma$.   
Suppose by contradiction that there is an element $r$ in $[n] \setminus (P \cup V)$ that lies in $S$. Since $\sigma$ has no double descents, $r$ must belong to an ascending segment, i.e.,
the $r$-factorization of $\sigma$ is then
   $w_1w_2 r \emptyset w_5$ 
   and 
    \[\varphi_r(\sigma)=w_1 \emptyset r w_2 w_5.\]
     In $\varphi_r(\sigma)$ we solely moved $r$ from  an ascending segment to a descending segment and left the rest of $\sigma$ unchanged. Applying the dual Foata-Strehl action at any other element of $[n] \setminus (P \cup V)$ will simply move an element from an ascending segment to a descending segment, hence it will not remove the double descent created in $\varphi_r(\sigma)$. Thus, $\pi=\varphi_S(\sigma)$ will contain a double descent, which contradicts the fact it is $FS$-minimal. We conclude that  $S\subset V$ and $\pi=\sigma_S(\sigma)=\sigma$.
\end{proof}

\section{Constructing and counting permutations with a fixed pinnacle set} \label{sec:count}

In this section we count the number of dual Foata-Strehl orbits with permutations having pinnacle set $P$ by counting the number of FS-minimal permutations with pinnacle set $P$ in $\sS_n$. Recall that a pair of sets $(P,V)$ is considered {admissible} if there is a permutation with pinnacle set $P$ and vale set $V$. Given an admissible tuple $(P,V)$ and a fixed integer $k \in [n]$, we set the following notation:

\begin{itemize}
\item Given a nonempty word $w$ of some letters in $[n]$, let $\max(w)$ be the largest number that appears in the word $w$.
\item Let $V_k = \{v \in V : v < k\}$. 
\item Let $P_k = \{p \in P : p < k\}$.
\item Let $\nPV(k) = |V_k| - |P_k|$, counting the number of vales less
than $k$, minus the number of pinnacles less than $k$.
\end{itemize}

\begin{lemma}\label{lem:admissible}
If $\pi$ is a permutation with pinnacle set $P=\{p_1<p_2<\cdots<p_\ell\}$ and vale set $V=\{v_1<v_2<\cdots<v_{\ell+1}\}$, then for all $1\leq i\leq \ell-1$, 
\begin{enumerate}[(a)]
    \item $1 \in V$, \label{part:1inV}
    \item $|V|=|P|+1$, \label{part:V=P+1}
     \item $2\leq \nPV(p_{i}),$\label{part:i}
     \item $v_{i+1}<p_i$ for all $i \in\{1,2,\dots,\ell\}$, \label{part:newaddon}
      \item $\nPV(p_{i})\leq \nPV(p_{i+1})+1,$\label{part:iandi+1}
    \item $\nPV(p_\ell)=2.$ \label{part:largestp}
\end{enumerate}
Furthermore, properties \eqref{part:i} and \eqref{part:newaddon} are equivalent.
\end{lemma}
\begin{proof}
For part \eqref{part:1inV}, since $1$ appears in $\pi$, we must have $\pi_i=1$ for some $i \in [n]$. Since $\pi_{i-1}>1<\pi_{i+1}$ then $1 \in V$. 

For part \eqref{part:V=P+1}, by definition the order of the pinnacles and vales alternates from vale to pinnacle and ends with a vale. Hence the claim follows.

For part \eqref{part:i}, let $p_i$ be any pinnacle in $P$. Consider the set  $P'=\{p_1,\cdots, p_{i-1},p_i\}\subseteq P$. Because each pinnacle has a vale smaller than it to its left and one to its right, and there is a vale between any two pinnacles, then there are at least $i+1$ vales (those around the $i$ pinnacles in $P'$) smaller than $p_i$ in $\pi$. Thus,
$$\nPV(p_i)=|V_{p_i}|-|P_{p_i}|\geq (i+1)-(i-1)=2. $$

For part \eqref{part:newaddon}, since $\nPV(p_i)=|V_{p_i}|-|P_{p_i}|\geq 2$ and $P_{p_i}=i-1$ for each $i$ in $\{1,2,\dots,\ell\}$, then $V_{p_i}\geq i+1$. That is, there are at least $i+1$ vales less than $p_i$. Hence, $v_{i+1}<p_i$.

For part \eqref{part:iandi+1}, note that since $p_{i+1}>p_i$, we must have $|V_{p_{i+1}}|\geq |V_{p_{i}}|$. Since there are $j-1$ pinnacles smaller than $p_j$ for any $p_j \in P$, we get $|P_{p_{i}}|=i-1$ and $|P_{p_{i+1}}|=i$. Thus, 
    $$\nPV(p_{i+1})=|V_{p_{i+1}}|-|P_{p_{i+1}}|\geq |V_{p_i}|-(|P_{p_i}|+1)=|V_{p_i}|-|P_{p_i}|-1=\nPV(p_{i})-1.$$

For part \eqref{part:largestp}, since $v_{\ell+1}<p_\ell$ by property \eqref{part:newaddon}, then $V_{p_\ell}=V$ and 
$$\nPV(p_\ell)=|V_{p_\ell}|-|P_{p_\ell}|=(\ell+1)-(\ell-1)=2.$$

To prove the last statement note that we already showed $\eqref{part:i}\Rightarrow\eqref{part:newaddon}$. For the reverse, if $v_{i+1}<p_i$ then $\nPV(p_i)=|V_{p_i}| - |P_{p_i}|\geq i+1-(i-1)=2.$
\end{proof}

 We now describe which pairs $(P,V)$ are admissible.

\begin{proposition}\label{prop:PVadmissible}
A pair $(P,V)$ is admissible if and only if properties $(a)$, $(b)$ and $(c)$ (or properties $(a)$, $(b)$ and $(d)$) from Lemma \ref{lem:admissible} hold.
\end{proposition}
\begin{proof}
The forward direction is proven in Lemma \ref{lem:admissible}. For the backward direction, suppose $P$ and $V$ satisfy properties $(a)$, $(b)$ and $(c)$ from Lemma \ref{lem:admissible}. Thus, $P$ and $V$ can be written as 
$P=\{p_1<p_2<\cdots<p_\ell\}$ and $V=\{v_1<v_2<\cdots<v_{\ell+1}\}$ with $v_1=1$. We need to create a permutation with pinnacle set $P$ and vale set $V$. 

Let $n=p_\ell$. Let $\alpha$ be defined as follows:
\[\alpha=v_1\, a_1\, p_1\, v_2\, a_2\, p_2\,  \cdots\, v_\ell\, a_{\ell}\, p_\ell\, v_{\ell+1}. \]
where each $a_i$ is an ascending sequence containing the elements in $[n]\setminus (P\cup V)$ between $v_{i}$ and $v_{i+1}$. Since properties $(c)$ and $(d)$ are equivalent, then $v_{i+1}<p_i$ for all $i \in \{1,2,\dots,\ell\}$. Thus, $\alpha$ has pinnacle set $P$ and vale set $V$.
\end{proof}

\subsection{Creating and counting the number of permutations with a fixed pinnacle and vale set} \label{subsec:create}

Given an admissible pair $(P,V)$ we define a \textbf{$PV$-arrangement}, denoted $\alpha_{PV}$, to be a permutation of the elements of $P\cup V$ such that every element $p \in P$ is a pinnacle in $\alpha_{PV}$ and every element $v\in V$ is a vale in $\alpha_{PV}$. We say that a $PV$-arrangement $\alpha_{PV}$ is a \textbf{canonical} if for each $p \in P$ the $p$-factorization $w_1w_2pw_4w_5$ of $\alpha_{PV}$ satisfies $\max(w_2) <\max(w_4)$.

\begin{lemma}[Counting Canonical $PV$-Arrangements]\label{lemma:cannonical}
For an admissible pair $(P,V)$ the number of canonical $PV$-arrangements is 
\begin{equation} \prod_{p \in P}\binom{\nPV(p)}{2}. \label{eqn:count} \end{equation}
\end{lemma}

\begin{proof}
We prove the result by induction on $|P|$. If $P=\emptyset$ then $V$ is a one element set $V=\{v\}$. In this case the only $PV$-arrangement is $\alpha=v$, which is canonical and is counted by the empty product in $\eqref{eqn:count}$. For completeness, we show the case $|P|=1$. If $P=\{p\}$ then $V$ is a set with two elements by Lemma \ref{lem:admissible}$(\ref{part:V=P+1})$, so $V=\{v_1,v_2\}$ for two elements $v_1,v_2 \in [n-1]$. Without loss of generality, let $v_1<v_2$. Then, the only $PV$-arrangements one could make are $\alpha_1=v_1pv_2$ and $\alpha_2=v_2pv_1$, of which only $\alpha_1$ is a canonical $PV$-arrangement. By Lemma \ref{lem:admissible}$(\ref{part:largestp})$, the product in $\eqref{eqn:count}$ is $\binom{\nPV(p)}{2}=\binom{2}{2}=1$, as desired.

Suppose the result is true for all pinnacle sets $P$ with cardinality $\ell-1$. Then if $|P|=\ell$, write $P=\{p_1, \ldots, p_{\ell}\}$ with $p_1<p_2<\cdots<p_\ell$. Choose any two elements $v_1,v_2 \in V$ such that $v_1<p_1$ and $v_2<p_1$ and let 
$$P'=P\setminus\{p_1\}\quad \text{ and } \quad V'=(V\setminus\{v_1,v_2\})\cup \{p_1\}.$$ 
Note that there are 
$$\binom{|V_{p_1}|}{2}=\binom{\nPV(p_1)}{2}$$ choices of $v_1$ and $v_2$. By Lemma \ref{lem:admissible}$(\ref{part:i})$, the number of choices is always at least 1.

For each such choice and for each canonical $P'V'$-arrangement, we will create one canonical $PV$-arrangement. By induction, this would imply that the number of canonical $PV$-arrangements is
\begin{equation}\label{eq:npv}\binom{\nPV(p_1)}{2}\prod_{p \in P'}\binom{\nPV[P'V'](p)}{2}.
\end{equation}
For $p\in P'$, 
$$\nPV[P'V'](p)=|V'_p|-|P'_p|=(|V_p|-1)- (|P_p|-1)=|V_p|-|P_p|=\nPV(p).$$
Thus, equation $\eqref{eq:npv}$ gives us
$$\binom{\nPV(p_1)}{2}\prod_{p \in P'}\binom{\nPV[P'V'](p)}{2}=\binom{\nPV(p_1)}{2}\prod_{p \in P'}\binom{\nPV(p)}{2}=\prod_{p \in P}\binom{\nPV(p)}{2}.$$

It remains to show that for each choice of $v_1,v_2\in V_{p_1}$ and for each canonical $P'V'$-arrangement, we can create a unique canonical $PV$-arrangement. To start, suppose without loss of generality that $v_1<v_2$. Let $\alpha_{P'V'}$ be a canonical $P'V'$-arrangement (hence,
$\Pin(\alpha_{P'V'})=P'$ and $\V(\alpha_{P'V'})=V'$). 
Thus, the element $p_1\in V'$ is a vale in $\alpha_{P'V'}$, so 
$$\alpha_{P'V'}=\cdots \,v_{k}\, p_{i}\,p_1\, p_{j}\, v_{k+1}\,\cdots, $$
for some pinnacles $p_i,p_j \in P'$ and vales $v_{k}, v_{k+1} \in V'$.
Insert $v_1$ to the left of $p_1$ and $v_2$ to the right of $p_1$ to create the permutation $\alpha_{PV}$
$$\alpha_{PV}=\cdots v_{k}\, p_{i}\,v_1\,p_1\,v_2\, p_j\, v_{k+1}\,\cdots.$$
 Note that $\alpha_{PV}$ is a $PV$-arrangement, that is, a permutation of $P\cup V$ with pinnacle set $P$ and vale set $V$. We claim this permutation is a canonical $PV$-arrangement. Since $\alpha_{P'V'}$ was a canonical $P'V'$-arrangement and we only added two numbers $v_1,v_2$ that are less than $p_1$ and adjacent to $p_1$, then $\max{w_2}$ is still less than $\max{w_4}$ in the $p'$-factorization of $\alpha_{PV}$ for any $p' \in P'$. For the $p$-factorization of $\alpha_{PV}$, note that $w_2=v_1$ and $w_4=v_2$ because all other peaks, particularly $p_i$ and $p_j$ are greater than $p_1$. Hence, $\max(w_2)<\max(w_4)$ in this factorization as well. Thus, $\alpha_{PV}$ is a canonical $PV$-arrangement.

To finish the proof, we need to show that all canonical $PV$-arrangements are created in this manner. For any canonical $PV$-arrangement $\alpha'_{PV}$, it must contain a subsequence $v_ip_1v_j$ with $v_i,v_j \in V_{p_1}$ and $v_i<v_j$. If we remove $v_i$ and $v_j$ from $\alpha'_{PV}$, we get a canonical $P'V'$-arrangement $\alpha'_{P'V'}$. Hence, $\alpha'_{PV}$ was obtained via our construction by choosing $v_i,v_j$ from $V_{p_1}$ and inserting $v_i$ and $v_j$ before and after $p_1$, respectively. Thus, our construction gives all canonical $PV$-arrangements.
\end{proof}

We are ready to prove the main theorem of the section. In it, we count the number of $FS$-minimal permutations with a given fixed pinnacle set $P$. By Theorem \ref{thm:uniqueFSminimal}, this also counts the number of dual Foata-Strehl orbits containing permutations with pinnacle set $P$.
\begin{thm}[FS-minimal permutations for each PV-arrangement] \label{thm:count}
For an admissible pair $(P,V)$, given a canonical $PV$-arrangement $\alpha_{PV}$, the number of FS-minimal permutations $\pi$ with $\pi|_{P\cup V}=\alpha_{PV}$, denoted $O_{PV}$, is
\[O_{PV} = \prod_{r \in [n] \setminus (P \cup V)} \nPV(r).\]
Furthermore, the number of all FS-minimal permutations with pinnacle set $P$, denoted $O_P$, is
\[O_P= \sum_{V\in \VV(P)}\prod_{p \in P}\binom{\nPV(p)}2\prod_{r \in [n] \setminus (P \cup V)} \nPV(r),\]
where $\VV(P)$ is the set of all vale sets $V$ for which the pair $(P,V)$ is admissible.
\end{thm}

\begin{proof}
Let $\alpha_{PV}=v_1p_1v_2p_2 \cdots v_\ell p_\ell v_{\ell+1}$ be a canonical $PV$-arrangement. Since in a $FS$-minimal permutation each pinnacle is immediately followed by a vale, to count the number of FS-minimal permutations $\pi$ with $\pi|_{P\cup V}=\alpha_{PV}$, note that each element $r \in ([n]\setminus(P \cup V))$ must appear to the right of a vale $v_i$ less than $r$ and to the left of a pinnacle $p_i$ greater than $r$.    The number of such indices $i$ so that $r$ satisfies $v_i<r<p_i$ is precisely $|V_r|-|P_r| =\nPV(r)$. The total number of choices over all $r \in ([n]\setminus(P \cup V))$ is then $$\prod_{r \in [n] \setminus (P \cup V)} \nPV(r).$$
 
Once these choices are made, all of the elements (if any)  between each vale $v_i$ and pinnacle $p_i$ must appear in ascending order by the definition of an $FS$-minimal permutation. This proofs the first statement. The last statement follows by summing through all the canonical $PV$-arrangements and using Lemma \ref{lemma:cannonical}.
\end{proof}

\begin{ex} 
A canonical arrangement $\alpha_{PV}$ for $P=\{7,10, 12\}$ and $V=\{1,3,5,8 \}$ is depicted in Figure \ref{fig:cannonical}.  To construct FS-minimal permutations from $\alpha_{PV}$ we insert each element $r \in [n] \setminus (P \cup V) = \{ 2,4,6,9,11 \}$ in ascending order on the slopes between $v_i$ and $p_i$ satisfying $v_i<r<p_i$.
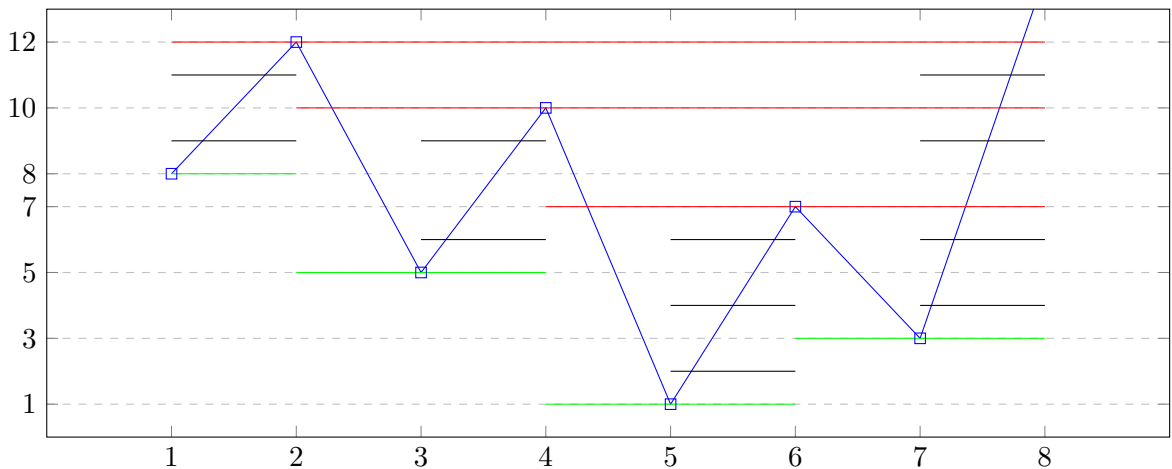
\begin{figure}[h]
\centering
\begin{tikzpicture}
\begin{axis}[
    height=\axisdefaultheight,
	width=\textwidth,
    title={}, 
    xtick={1,2,3,4,5,6,7,8},
    ytick={1,3,5,7,8,10,12},
    ymajorgrids=true,
    grid style=dashed,
    xmin = 0,
    xmax = 9,
    ymin = 0,
    ymax = 13
]

\addplot[color=blue,mark=square]coordinates{(1,8)(2,12)(3,5)(4,10)(5,1)(6,7)(7,3)(8,14)};
\addplot[color=green,domain=1:2]{8};
\addplot[color=red,domain=1:8]{12};
\addplot[color=green,domain=2:4]{5};
\addplot[color=red,domain=2:8]{10};
\addplot[color=green,domain=4:6]{1};
\addplot[color=red,domain=4:8]{7};
\addplot[color=green,domain=6:8]{3};
\addplot[color=black,domain=5:6]{2};
\addplot[color=black,domain=5:6]{4};
\addplot[color=black,domain=7:8]{4};
\addplot[color=black,domain=3:4]{6};
\addplot[color=black,domain=5:6]{6};
\addplot[color=black,domain=7:8]{6};
\addplot[color=black,domain=1:2]{9};
\addplot[color=black,domain=3:4]{9};
\addplot[color=black,domain=7:8]{9};
\addplot[color=black,domain=1:2]{11};
\addplot[color=black,domain=7:8]{11};
\end{axis}
\end{tikzpicture}
\caption{A canonical arrangement $\alpha_{PV} = [8,12,5,10,1,7,3]$.}\label{fig:cannonical}
\end{figure}

Table \ref{tab:ex cannonical} describes the possible locations where the elements $r \in [n] \setminus (P \cup V) = \{ 2,4,6,9,11 \}$ can be located in a FS-minimal permutation. For example, if we choose the first possible location listed in the third column in Table \ref{tab:ex cannonical}, the resulting FS-minimal permutation is $[8,9,11,12,5,6,10,1,2,4,7,3]$.

\begin{table}[h]
    \centering
\begin{tabular}{|c|c|l|} \hline  
$r$& $N_{PV}(r)$ & Possible locations of $r$ in FS-minimal permutation\\  \hline\hline
2 & 1 & between $(1,7)$  \\
4 & 2 & between $(1,7)$ or right of $3$ \\
6 & 3 & between $(5,10)$, between $(1,7)$, or right of $3$ \\
9 & 3 & between $(8,12)$, between $(5,10)$, or right of $3$ \\
11 & 2 & between $(8,12)$ or right of $3$ \\  \hline
\end{tabular} 
\caption{Example of possible locations of where to insert certain values to create an FS-minimal permutation.}
    \label{tab:ex cannonical}
\end{table}
Hence, in this case, there are $2\cdot 3  \cdot 3 \cdot 2 = 36$ FS-minimal permutations $\pi$ with $\pi|_{P\cup V} = \alpha_{PV}$.

\end{ex}

\begin{cor}\label{cor:main result}
If $P$ is an admissible pinnacle set, then 
\[|\Pin(P;n)|=2^{n-|P|-1}\cdot O_P=2^{n-|P|-1}\left(\sum_{V\in \VV(P)}\prod_{p \in P}\binom{\nPV(p)}2\prod_{r \in [n] \setminus (P \cup V)} \nPV(r)\right).\]
\end{cor}
\begin{proof}
By Lemma \ref{thm:orbitsize}, each orbit of the Foata Strehl action has $2^{n-|P|-1}$ elements. By Theorem~ \ref{thm:count}, the number of orbits containing permutations with pinnacle set $P$, denoted $O_P$, is
\[O_P=\sum_{V\in \VV(P)}\prod_{p \in P}\binom{\nPV(p)}2\prod_{r \in [n] \setminus (P \cup V)} \nPV(r).\]
Hence, we obtain our result.
\end{proof}

Since computing the number of permutations with a given pinnacle set $P$ depends on the number of admissible vale sets, we now construct and establish account for the number of vale sets given~$P$.

\subsection{Creating and counting the number of admissible vale sets}
Given an nonnegative integer $\ell$, we recall that a weak composition is a sequence of nonnegative integers whose sum is $\ell$. For $\ell\geq 1$, define
$$C(\ell)= \left  \{ (t_1,t_2,\ldots, t_\ell) \vdash \ell \,\, \Big{|}\,\, \sum_{i=1}^k t_i \geq k  \text{ for each $k$ in } \{1,2,\dots,\ell\} \right\} $$
and 
$C(0)=\{()\}$ just contains the empty tuple.
Given a pinnacle set $P=\{p_1<p_2<\cdots< p_\ell\}$, for $1\leq i\leq \ell$, let $G_i = \{j \in [n]\setminus P : p_{i-1} < j < p_i\},$ where $p_0=1$. We call $G_i$  the $i$th gap set. Let $g_i = |G_i|$. For example, if $P=\{4,8, 11\}$ then $G_1=\{2,3\}, G_2=\{5,6,7\}$ and $G_3=\{9,10\}$. We now proceed to create and count all possible vale sets given a fixed pinnacle set.\\
\begin{proposition}\label{prop:valesets}
If $P=\{p_1<p_2<\ldots< p_{\ell}\}$ is an admissible pinnacle set, then the number of sets $V$ such that $(P,V)$ is admissible is
$$|\mathcal{V}(P)|=   \sum_{t \in C(\ell)}\prod_{i=1}^\ell \binom{g_i}{t_i}. $$
\end{proposition}
\begin{proof} 
For each weak composition $t=(t_1,t_2,\ldots, t_\ell) \in C(\ell)$, we will construct $\prod_{i=1}^\ell \binom{g_i}{t_i}$ admissible vale sets. Recall for $1\leq i\leq \ell$, we define $G_i = \{j \in [n]\setminus P : p_{i-1} < j < p_i\}$ where $p_0=1$ and $g_i=|G_i|$.
By definition note $G_i\cap G_j=\emptyset$ whenever $i\neq j$.

For each $G_i$, we choose a subset $T_i\subseteq G_i$ of cardinality $t_i$, recalling that for each $1\leq i\leq \ell$, we have that $\sum_{i=1}^kt_i\geq k$, for all $k\in \{1,\ldots, \ell\}$. There are $\prod_{i=1}^\ell\binom{g_i}{t_i}$ different choices for the collection of subsets $T_i\subseteq G_i$.

Once all the $T_i$ are chosen, let $V = \{1\} \cup \bigcup_{i = 1}^\ell T_i$.  Since $T_i\cap T_j=\emptyset$ then  
$|V|=1+\sum_{i=1}^\ell{t_i}=1+\ell$. 
We now prove that $(P,V)$ is  admissible by creating a permutation with pinnacle set $P$ and vale set $V$.
Sort $V$ such that $V=\{v_0,v_1,v_2,\dots, v_\ell\}$ with $v_0=1$ and $v_{i-1}<v_i$ for $i\in\{1,\dots,\ell\}.$ Let 
$$\alpha=v_0p_1v_1p_2v_2\cdots p_\ell v_\ell.$$
Note that $\alpha$ is a permutation of $V\cup P$ with pinnacle set $P$ and vale set $V$ because for any pinnacle $p_k$, by the defining condition of $C(\ell)$, we have chosen $\sum_{i=1}^k t_i$ elements in $V$, all of which are less than $p_k$. Thus, the vales $1,v_1,\ldots, v_k$ are all less than $p_k$, which implies that the pinnacle set (resp. vale set) of $\alpha$ is $P$ (resp. $V$). To complete the claim, extend $\alpha$ to a permutation of $[n]$ by inserting each element $r \in [n]\setminus(P\cup V)$  in ascending order between $v_i$ and $p_{i}$ if $ v_i< r<p_{i+1}$.    The resulting permutation $\pi$ has pinnacle set $P$ and vale set $V$. Hence, $(P,V)$ is admissible.

{To show these are the only admissible pairs, let $V'$ be any set such that $(P,V')$ is admissible. Then there is a permutation $\pi$ such that $\Pin(\pi)=P$ and $\textup{Vale}(\pi)=V'$. Partition $V'$ as $V'=\{1\}\cup \bigcup_{i=1}^\ell T'_i$ where for $1\leq i\leq \ell$, $T'_i$ is defined as
$$T'_i = \{v \in V' : p_{i-1} < v < p_i\}, \text{ with $p_0=1$.}$$
Note that, by definition of $G_i$, we have $T'_i \subseteq G_i$ and let  $t_i':=|T_i'|$.
Since $T_i'\cap T_j'=\emptyset$ whenever $i\neq j$ and  $\cup_{i=1}^\ell T'_i=V\setminus\{1\}$, we have that $\sum_{i=1}^{\ell}t_i'=|V|-1=\ell$. Thus, $\textbf{t}=(t_1',\ldots,t_\ell')$ is a weak composition of $\ell$. 
Further, $\textbf{t}$ is an element of $C(\ell)$ since for each $k$ in $\{1,2,\dots, \ell\}$ the elements in $\cup_{i=1}^k T'_i$ correspond to choices of vales not equal to 1 that are less than $p_k$, thus
$$\sum_{i = 1}^k t_i' = |V_{p_k}| - 1 = |V_{p_k}| -|P_{p_k}| + (k - 1) - 1 = k - 2 + \nPV(p_k) \geq k,$$
where the last inequality follows from Lemma \ref{lem:admissible}(\ref{part:i}). Hence, the set $V'$ is created via the construction described in this result by starting with $t \in C(\ell)$ and making the choices so that $T_i=T'_i\subseteq G_i$. Therefore, we have constructed all vale sets $V$ such that $(P,V)$ is admissible.
}
\end{proof}

We now give two examples to compute the number of permutations with pinnacle sets $P=\{5\}$ and $P=\{4,8,11\}$.
\begin{ex}
Let $n=8$ and consider the admissible pinnacle set $P=\{5\}$. As one is always a vale, and $G_1=\{2,3,4\}$ the possible vale sets are $\VV(P)=\{\{1,2\},\{1,3\},\{1,4\}\}$. Since $\nPV(5)=2$, by Corollary \ref{cor:main result} we have that 
\begin{align*}
    |\Pin(P;10)|&=2^6\left(\sum_{V\in \VV(P)}\binom{\nPV(5)}2\prod_{r \in [n] \setminus (P \cup V)} \nPV(r)\right)\\
    &=2^{6}\left[\underbrace{\binom{2}2\prod_{r \in \{3,4,6,7,8\}} \nPV(r)}_{V=\{1,2\}}+\underbrace{\binom{2}2\prod_{r \in \{2,4,6,7,8\}} \nPV(r)}_{V=\{1,3\}}+\underbrace{\binom{2}2\prod_{r \in \{2,3,6,7,8\}} \nPV(r)}_{V=\{1,4\}}\right]\\
    &=2^{6}\left[1\cdot2\cdot2\cdot1\cdot1\cdot1\cdot1+1\cdot1\cdot2\cdot1\cdot1\cdot1+1\cdot1\cdot1\cdot1\cdot1\cdot1\right]\\
    &=2^6(2^2+2+1)\\
    &=448.
\end{align*}
\end{ex}

\begin{ex}\label{example:vp}
Let $n=12$ and consider $P=\{4,8,11\}$ so $\ell=3$. 
Then $$C(\ell) = \{(1,1,1), (2,0,1),(2,1,0), (1,2,0), (3,0,0) \}.$$ The gaps are $G_1=\{2,3\}$, $G_2=\{ 5,6,7\}$, and $G_3=\{ 9,10\}$, so $(g_1,g_2,g_3)= (2,3,2)$.
The number of admissible vale sets 
is $$ |\mathcal{V}(P)| = \binom{2}{1}\binom{3}{1}\binom{2}{1}+\binom{2}{2}\binom{3}{0}\binom{2}{1}+\binom{2}{2}\binom{3}{1}\binom{2}{0}+\binom{2}{1}\binom{3}{2}\binom{2}{0} + \binom{2}{3}\binom{3}{0}\binom{2}{0}=23.$$
Note that the term
$\binom{2}{1}\binom{3}{1}\binom{2}{1}$
counts the vale sets with one element coming from each set of gaps, while the term $\binom{2}{2}\binom{3}{0}\binom{2}{1}$ counts the vale sets where 2 elements come from the first set of gaps, 0 come from the second set of gaps, and 1 comes from the third set of gaps, etc. Thus 
\begin{align}
    \VV(P)=\left\{\begin{matrix}\{1,2,5,9\},\{1,2,5,10\},\{1,2,6,9\},\{1,2,6,10\},\{1,2,7,9\},\{1,2,7,10\},\\\{1,3,5,9\},\{1,3,5,10\},\{1,3,6,9\},\{1,3,6,10\},\{1,3,7,9\},\{1,2,7,10\},\\
    \{1,2,3,9\},\{1,2,3,10\},\{1,2,3,5\},\{1,2,3,6\},\{1,2,3,7\},\\
    \{1,2,5,6\},\{1,2,5,7\},\{1,2,6,7\},\{1,3,5,6\},\{1,3,5,7\},\{1,3,6,7\}
    \end{matrix}\right\}.
\end{align}
From this a straight forward computation using Corollary \ref{cor:main result} yields $|\Pin(\{4,8,11\};12)|=132,480$.
\end{ex}

\begin{remark}
The partitions $t\in C(\ell)$ indexing the sum in Proposition \ref{prop:valesets} are counted by Catalan numbers. Particularly, $|C(\ell)|$ is the $\ell$th Catalan number. One way to see the equivalence is to create a bijection between Dyck paths and the elements of $C(\ell)$. 
A Dyck path is a path from $(0,0)$ to $(n,n)$ in which you are only allowed to move right and up and that lies strictly below (but may touch) the diagonal $y=x$. Given any $t=(t_1,t_2,\dots, t_\ell) \in C(\ell)$ consider the Dyck path created by moving up $t_i$ times then once to the right.  For instance, the 5 Dyck paths created from the 5 elements of $C(\ell)$ in Example \ref{example:vp} are:
$$\begin{array}{ccc}
    (1,1,1) &\to& \text{URURUR} \\
    (2,0,1) &\to& \text{UURRUR} \\
    (2,1,0) &\to& \text{UURURR} \\
    (1,2,0) &\to& \text{URUURR}  \\
    (3,0,0) &\to& \text{UUURRR}.
\end{array}$$
It is left as an exercise to the reader to show that this map is a bijection. 
\end{remark}

\section{Algorithms to generate all permutations with a given pinnacle set}\label{sec:code}

As we saw in Section \ref{sec:FSfacts}, the dual Foata-Strehl action preserved the pinnacles of a permutation, but the orbits did not encompass all elements having the same set of pinnacles.  
In this section, we describe two algorithms that generate the set $
\Pin(P;n)$, and we compare their computational run times.

Given a pinnacle set $P$, let Algorithm 1 be the naive algorithm that runs through all permutations of $S_n$, computes their pinnacle sets and returns those permutations with pinnacle set $P$.
Let Algorithm 2 be the algorithm that replicates the constructions detailed in Section \ref{sec:count}. More specifically, given a set $P$, it first runs through all admissible vale sets $V$ using the criteria in Proposition \ref{prop:PVadmissible}. Then, for a given pair $(P,V)$, it constructs all canonical $PV$-arrangements using the recursive construction described in the proof of Lemma \ref{lemma:cannonical}. Then, it creates all FS-minimal permutations from the canonical PV-arrangements as described in the proof of Theorem \ref{thm:count}. Finally, it applies the dual Foata-Strehl action on the FS-minimal permutations to create all permutations with pinnacle set $P$ and vale set $V$, as guaranteed by Theorem \ref{thm:uniqueFSminimal}.

In Table \ref{tab:faster}, we provide the run times of Algorithms 1 and 2 applied to all pinnacle sets of permutations in $\mathfrak{S}_8$. The  code and sample computations for these algorithms is provided at \textcolor{blue}{\href{https://github.com/8080509/Pinnacles_of_Permutations}{github.com/8080509/Pinnacles\_of\_Permutations}}.
\begin{table}
\centering
\begin{tabular}{|c|c|c|c|c|}\hline
    \multirow{2}{*}{$n$}& \multirow{2}{*}{$P$}& \multirow{2}{*}{$|\Pin(P;n)|$} & Run time& Run time\\
    &&&Algorithm 1&Algorithm 2
    \\\hline\hline
    \multirow{28}{*}{8}&$\emptyset$& 128 & 327.32 ms & 0.30 ms \\\cline{2-5}
    &\{3\} & 64 & 286.39 ms & 0.21 ms  \\\cline{2-5}
    &\{4\} & 192 & 300.00 ms & 0.63 ms \\\cline{2-5}
    &\{5\} & 448 & 346.05 ms & 1.12 ms \\\cline{2-5}
    &\{6\} & 960 & 360.80 ms & 2.65 ms \\\cline{2-5}
    &\{7\} & 1984 & 293.21 ms & 6.78 ms \\\cline{2-5}
    &\{8\} & 4032 & 271.34 ms & 9.45 ms \\\cline{2-5}
    &\{3, 5\} & 32 & 411.87 ms & 0.11 ms  \\\cline{2-5}
    &\{3, 6\} & 96 & 480.09 ms & 0.54 ms  \\\cline{2-5}
    &\{3, 7\} & 224 & 436.53 ms & 0.59 ms \\\cline{2-5}
    &\{3, 8\} & 480 & 275.81 ms & 1.13 ms \\\cline{2-5}
    &\{4, 5\} & 96 & 309.15 ms & 0.43 ms \\\cline{2-5}
    &\{4, 6\} & 288 & 306.61 ms & 1.31 ms \\\cline{2-5}
    &\{4, 7\} & 672 & 280.14 ms & 1.64 ms \\\cline{2-5}
    &\{4, 8\} & 1440 & 291.30 ms & 3.67 ms \\\cline{2-5}
    &\{5, 6\} & 576 & 324.70 ms & 1.66 ms \\\cline{2-5}
    &\{5, 7\} & 1376 & 307.15 ms & 3.75 ms \\\cline{2-5}
    &\{5, 8\} & 2976 & 285.69 ms & 10.77 ms \\\cline{2-5}
    &\{6, 7\} & 2400 & 297.79 ms & 6.08 ms \\\cline{2-5}
    &\{6, 8\} & 5280 & 338.10 ms & 14.92 ms \\\cline{2-5}
    &\{7, 8\} & 8640 & 341.82 ms & 20.45 ms \\\cline{2-5}
    &\{3, 5, 7\} & 16 & 298.76 ms & 0.13 ms \\\cline{2-5}
    &\{3, 5, 8\} & 48 & 282.71 ms & 0.19 ms \\\cline{2-5}
    &\{3, 6, 7\} & 48 & 269.83 ms & 0.20 ms \\\cline{2-5}
    &\{3, 6, 8\} & 144 & 328.53 ms & 0.53 ms \\\cline{2-5}
    &\{3, 7, 8\} & 288 & 305.97 ms & 0.84 ms \\\cline{2-5}
    &\{4, 5, 7\} & 48 & 296.36 ms & 0.20 ms  \\\cline{2-5}
    &\{4, 5, 8\} & 144 & 341.79 ms & 0.71 ms \\\cline{2-5}
    &\{4, 6, 7\} & 144 & 294.63 ms & 0.47 ms \\\cline{2-5}
    &\{4, 6, 8\} & 432 & 296.74 ms & 1.31 ms \\\cline{2-5}
    &\{4, 7, 8\} & 864 & 358.77 ms & 2.68 ms \\\cline{2-5}
    &\{5, 6, 7\} & 288 & 294.47 ms & 0.99 ms \\\cline{2-5}
    &\{5, 6, 8\} & 864 & 305.26 ms & 2.44 ms \\\cline{2-5}
    &\{5, 7, 8\} & 1728 & 334.79 ms & 4.98 ms \\\cline{2-5}
    &\{6, 7, 8\} & 2880 & 276.56 ms & 8.08 ms\\\hline
\end{tabular}
\caption{Run times of four algorithms constructing all permutations in $\mathfrak{S}_8$ with pinnacle set $P$.}\label{tab:faster}
\end{table}

\section{Future directions}\label{sec:future}
We end with a few open problems for further study. 

\begin{problem}
Algorithm 2 provides an efficient algorithm to generate $\Pin(P,n)$. Are there any other algorithms for generating $\Pin(P,n)$ that are more efficient than Algorithm 2.  \end{problem}

In \cite{pinn}, Davis et al. give explicit formulas for the number of permutations with  pinnacle sets of size 0, 1, and 2 as well as two extremal cases.
\begin{problem}\label{prob:1}
Theorem \ref{thm:count} provides an expression for the number $O_P$ of orbits containing permutations with pinnacle set $P$ under the dual Foata-Strehl action. Find explicit expressions for $\Pin(P,n)$ with $|P|\geq 3$ using Corollary \ref{cor:main result}.
\end{problem}
If $P$ is a pinnacle set and $S$ is a peak set, by Theorem~\ref{thm:BBS} and Corollary~\ref{cor:main result},
we know
\[|\Pin(P,n)|=2^{n-|P|-1}O_P\qquad\mbox{and}\qquad |\Pk(S,n)|=2^{n-|S|-1}p_S(n),\]
where $p_S(n)$ is the peak polynomial of $S$ and $O_P$ is given by 
\[O_P=\sum_{V\in \VV(P)}\prod_{p \in P}\binom{\nPV(p)}2\prod_{r \in [n] \setminus (P \cup V)} \nPV(r).\]
In light of the similarity between these equations and the fact that in the pinnacle setting the power of two describes the size of each dual Foata-Strehl orbit in $|\Pin(P,n)|$, and $O_P$ counts number of orbits, we pose the following question.

\begin{problem}
Is there a group action on permutations which preserves peaks sets, such that there are exactly $p_S(n)$ many orbits each of size $2^{n-|S|-1}$?
\end{problem}

We have presented the following conjecture at several talks concerning peaks, descents, and pinnacles of permutations over the past year.   An elegant proof of this conjecture was recently given in the preprint \cite[Corollary~ 10]{GaeGao}.  We present the conjecture here to have it recorded in the literature.

\begin{conjecture} 
If $S$ is an admissible peak set, then the set $\Pk(S,n)$ of permutations with peak set $S$ in 
$\mathfrak{S}_n$ can be partitioned into subsets of permutations of the same length, and the size of these subsets is palindromic about the value $\binom{n}{2}$.
\end{conjecture} 

We remark that as the sets $\Pin(P;n)$ are preserved by multiplying by $w_0$ on the left (reversing the order of the permutations), the sets $\Pin(P;n)$ have the same palindromicity property as $\Pk(S,n)$. However, unlike $\Pk(S;n)$, they are not unimodal or log-concave in general.

\begin{problem}
For what pinnacle sets $P$ and values of $n$ are the sets $\Pin(P;n)$ unimodal?
\end{problem}

\end{document}